\documentclass{article}
\usepackage[utf8]{inputenc}
\usepackage{amsmath, amssymb, verbatim, setspace}
\usepackage{amsthm}
\usepackage{tikz}
\usepackage{tikz-cd}
\usepackage{siunitx} 
\usepackage{graphicx}
\usepackage{pdflscape}

\newtheoremstyle{rmk}
  {\topsep}
  {\topsep}
  {}
  {0pt}
  {\bfseries}
  {.}
  { }
  {\thmname{#1}\thmnumber{ #2}\thmnote{ (#3)}}
\newtheorem{theorem}{Theorem}[section]

\newtheorem{corollary}[theorem]{Corollary}
\newtheorem{lemma}[theorem]{Lemma}

\theoremstyle{rmk}

\newtheorem{definition}[theorem]{Definition}
\newtheorem{note}[theorem]{Remark}
\begin{document}

\title{The Supersingularity of Hurwitz Curves} 
\author{Erin Dawson, Henry Frauenhoff, Michael Lynch, Amethyst Price,\\ Seamus Somerstep, Eric Work \\
		Graduate Student Advisor: Dean Bisogno\\ Faculty Advisor: Rachel Pries}
\maketitle
\makebox[\linewidth]{\rule{12cm}{0.4pt}}
\begin{abstract}
We study when Hurwitz curves are supersingular. Specifically, we show that the curve $H_{n,\ell}: X^nY^\ell + Y^nZ^\ell + Z^nX^\ell = 0$, with $n$ and $\ell$ relatively prime, is supersingular over the finite field $\mathbb{F}_{p}$ if and only if there exists an integer $i$ such that $p^i \equiv -1 \bmod (n^2 - n\ell + \ell^2)$. If this holds, we prove that it is also true that the curve is maximal over $\mathbb{F}_{p^{2i}}$. Further, we provide a complete table of supersingular Hurwitz curves of genus less than 5 for characteristic less than 37.

\begin{center}
Minimal Curve, Fermat Curve, Supersingular Curve
\end{center}

\end{abstract}
\makebox[\linewidth]{\rule{12cm}{0.4pt}}

\section{Introduction}
In 1941, Deuring defined the basic theory of supersingular elliptic curves. Supersingular curves are useful in error-correcting codes called Goppa codes. They also have potential applications to quantum resistant cryptosystems. 

In this paper we determine a condition for supersingularity of Hurwitz curves $H_{n,\ell}$ when $n$ and $\ell$ are relatively prime. In particular we show that every supersingular Hurwitz curve $H_{n,\ell}$ is maximal over some finite field. We also provide a classification of supersingular Hurwitz curves with genus less than 5 over fields with characteristic less than 37 and find some restrictions on the genera of Hurwitz curves. 

\subsection{Acknowledgements}

We would like to thank Dr.~Rachel Pries for proposing this question and guiding us through our research process. We would also like to thank Dr.~{\"O}zlem Ejder for all of her help and the referee for helpful comments. Finally, we would like to thank the College of Natural Sciences, the CSU Department of Mathematics, and the National Science Foundation for the REU supplement to DMS-15-02227. Pries was partially supported by NSF grant DMS-15-02227. This project would not be possible without all of you.

\section{Background information}

We first define the Hurwitz curve and the Fermat curve. Next we define the zeta function of a curve. From the zeta function we compute the normalized Weil numbers which we use to study supersingularity. We must also state the Hasse-Weil bound in order to define maximality and minimality.

\subsection{The Hurwitz curve and the Fermat curve}\label{sec:curves}
Let $n$, $\ell$, and $d$ be positive integers.  Let $F$ be a field.

\begin{definition}[Hurwitz curve $H_{n,\ell}$]
The {\em Hurwitz curve} $H_{n,\ell}$ over $F$ is given by the projective equation
\begin{equation*}
H_{n,\ell}: X^nY^\ell+Y^nZ^\ell+Z^nX^\ell=0.
\end{equation*}
\end{definition}
Throughout this paper, set $m=n^2-n\ell+\ell^2$. The Hurwitz curve $H_{n,\ell}$ has the following genus
\begin{equation*}
g=\frac{m+2-3\gcd{(n,\ell)}}{2}
\end{equation*} 
and is smooth when the characteristic $p$ of $F$ is relatively prime to $m$.

\begin{definition}[Fermat curve $\mathcal{F}_{d}$]
The {\em Fermat curve} of degree $d$ over $F$ is given by the projective equation
\begin{equation*}
\mathcal{F}_{d}: U^{d}+V^{d}+W^{d}=0.
\end{equation*}
\end{definition}
The Fermat curve $\mathcal{F}_d$ has genus $\frac{(d-1)(d-2)}{2}$ and is smooth when the characteristic $p$ of $F$ does not divide $d$.
Note that the Hurwitz curve $H_{n,\ell}$ is covered by the Fermat curve of degree $m = n^2 - n\ell+\ell^2$; see Section \ref{sec:covers} for more details.

\subsection{Zeta Function}

Let $\mathbb{F}_q$ be a finite field of cardinality $q$ where $q$ is a power of a prime $p$.
For a curve $C$ defined over $\mathbb{F}_q$, denote the number of points on $C$ by $\#C(\mathbb{F}_q)$. For extensions of $\mathbb{F}_q$, define $N_s = \#C(\mathbb{F}_{q^s})$. 
\begin{definition}[Zeta function]
The {\em zeta function} of a curve $C/\mathbb{F}_q$ is the series
\begin{equation} \label{Eq:Zeta}
{Z}(C/\mathbb{F}_q,T)=\text{exp}\bigg(\sum_{s=1}^{\infty}\frac{{N}_{s}T^s}{s}\bigg).
\end{equation}
\end{definition}
Rationality of the zeta function for curves was proven by Weil \cite{MR0027151, MR0029393}. 
In particular, Weil showed that the zeta function can be written as
\begin{equation} \label{Eq:Zeta-Weil}
{Z}(C/\mathbb{F}_q,T) = \frac{	L(C/\mathbb{F}_q,T)}{(1-T)(1-qT)}.
\end{equation}
The $L$-polynomial, $L(C/\mathbb{F}_q,T) \in \mathbb{Z}[T]$, has degree $2g$ \cite[page 152]{IrelandRosen},
\begin{equation} \label{Eq:L-poly}
L(C/\mathbb{F}_q,T)=1+C_1T+...+C_{2g}T^{2g}.
\end{equation}
The $L$-polynomial of a curve $C$ over $\mathbb{F}_q$ with genus $g$ factors in $\mathbb{C}[T]$ as
\begin{equation*}
L(C/\mathbb{F}_q,T) = \displaystyle\prod_{i=1}^{2g}(1-\alpha_iT).
\end{equation*}
Furthermore, $|\alpha_i| = \sqrt{q}$ for each $1 \leq i \leq 2g$ \cite[page 155]{IrelandRosen}. The normalized Weil numbers (NWNs) are the normalized reciprocal roots of the $L$-polynomial.
\begin{definition}[Normalized Weil Numbers]
The {\em Weil numbers} of $C/\mathbb{F}_q$ are the reciprocal roots $\alpha_i$ of $L(C/\mathbb{F}_q,T)$ for $1 \leq i \leq 2g$. The {\em normalized Weil numbers} are the values $\alpha_i/\sqrt{q}$ for $1 \leq i \leq 2g$.
\end{definition}
\begin{note} \label{Rmk:NWN/Extentions}
If $\{\alpha_{1}, \ldots ,\alpha_{2g}\}$ are the normalized Weil numbers over $\mathbb{F}_q$, then $\{\alpha_{1}^i, \ldots ,\alpha_{2g}^i\}$ are the normalized Weil numbers over $\mathbb{F}_{q^i}$.
\end{note}

The coefficients of $L(C/\mathbb{F}_q,T)$ follow a pattern.
For $k \in {\mathbb N}$, we denote the set of partitions of $k$ by $\text{par}(k)$ 
and the length of a partition $\gamma$ by $\text{len}(\gamma)$. 

\begin{lemma} \label{Lem:ZetaCoeffs}
In Equation (\ref{Eq:L-poly}) for $0 \leq k \leq 2g$, the coefficient $C_k$ has the form 
\begin{align*}
C_k=\sum_{\gamma \in {\rm par}(k)} \frac{\displaystyle \prod_{j \in \gamma}\frac{N_j}{j}}{\rm len(\gamma)!}-\sum_{i=0}^{k-1}(C_i \sum_{\mu =0}^{k-i}q^\mu).
\end{align*}
\end{lemma}

\begin{proof}
Equation \eqref{Eq:Zeta} can be expanded using the Taylor series of the exponential function
\begin{align*}
Z(C/\mathbb{F}_q,T)=\displaystyle \sum_{i=0}^{\infty}\frac{(N_1T+\frac{N_2}{2}T^2+\ldots+\frac{N_{2g}}{2g}T^{2g})^i}{i!}.
\end{align*}
Collecting terms up through $T^3$ gives a pattern to follow: 
\begin{equation}\label{eq:zetaside1}
Z(C/\mathbb{F}_q,T)=1 + (N_1)T + \bigg(\frac{N_2}{2}+\frac{N_1^2}{2}\bigg)T^2 + \bigg(\frac{N_3}{3}+\frac{N_1N_2}{2}+\frac{N_1^3}{6}\bigg)T ^3 + \ldots.
\end{equation}

The key step is to recognize that the subscripts on the $N_j$ are the partitions of $k$. The coefficient on $T^k$ can be written as
\begin{align*}
\sum_{\gamma \in \text{par}(k)} \frac{\displaystyle \prod_{j \in \gamma}\frac{N_j}{j}}{\text{len}(\gamma)!}.
\end{align*}

Equation \eqref{Eq:Zeta-Weil} gives a simplified version of $Z(C/\mathbb{F}_q,T)$. Using the Taylor series for each of the denominator terms as well as equation \eqref{Eq:L-poly} yields the following expansion: 
\begin{equation}\label{eq:zetaside2}
Z(C/\mathbb{F}_q,T)=(1+C_1T+...+C_{2g}T^{2g})(1+T+T^2+\ldots)(1+qT+q^2T^2+\ldots).
\end{equation}
Expanding and collecting terms, the coefficients on $T^k$ are given by
\begin{equation*} 
\sum_{i=0}^{k-1}(C_i \sum_{j =0}^{k-i}q^j)+C_k.
\end{equation*}
Setting equation \eqref{eq:zetaside1} and equation \eqref{eq:zetaside2} equal and comparing coefficients gives a linear system allowing one to solve for $C_k$ in terms of the values of $N_s$.
\end{proof}

\subsection{The Newton Polygon and Supersingularity}
Fix a curve $C/\mathbb{F}_q$ with associated $L$-polynomial $L(C/\mathbb{F}_q,T)$. 

\begin{definition}[Supersingularity] \label{Rmk:SS-RootsOfUnity}
The curve $C$ is {\em supersingular} if all its normalized Weil numbers are roots of unity.
\end{definition}

Another way to check if $C$ is supersingular is with its Newton polygon.

\begin{definition}[Normalized Valuation on $\mathbb{F}_{p^r}$]\label{defn:val}
Let $n = p^lk$ be an integer with $p \nmid k$. We denote the normalized $\mathbb{F}_{p^r}$-valuation of $n$ by $\text{val}_{p^r}(n) = \frac{l}{r}$ and the prime-to-$p$ part of $n$ by $n_p = k$. If $n=0$, we say $\text{val}_{p^r}(0) = \infty$.
\end{definition}
\begin{definition}[Newton Polygon]
Fix a curve $C/\mathbb{F}_{p^r}$ with $L$-polynomial in the form of equation \eqref{Eq:L-poly}. The {\em Newton polygon} of $C/\mathbb{F}_{p^r}$ is the lower convex hull of the points $\{(i,\text{val}_{p^r}(C_i)) \mid 0 \leq i \leq 2g \}$.
\end{definition}

\begin{note}\label{rmk:newton}
Because $C_0 = 1$ for every curve $C/\mathbb{F}_{p^r}$, the Newton polygon will always have initial point $(0,0)$. Likewise the final coefficient of $L(C/\mathbb{F}_{p^r}, T)$ is always $C_{2g} = p^{rg}$. For this reason the Newton polygon always has terminal point $(2g,g)$.
\end{note}
From Remark \ref{rmk:newton}, we can see that the Newton polygon of a curve $C$ over $\mathbb{F}_{p^r}$ is always a union of line segments on or below the line $y = \frac{1}{2}x$ with increasing slopes.
\begin{note}
A curve $C/\mathbb{F}_{q}$ is supersingular if and only if its Newton polygon is a line segment with slope $\frac{1}{2}$.
\end{note}

\subsection{Minimality and Maximality}
As a consequence of the Weil conjectures, the number of points on a curve $C/\mathbb{F}_q$ is controlled by the Hasse-Weil bound:
\begin{equation*}
1+q-2g\sqrt{q} \leq \#C(\mathbb{F}_{q}) \leq 1+q+2g\sqrt{q}.
\end{equation*}
The Hasse-Weil bound for curves was proven by Weil \cite{MR0027151}.

\begin{definition}[Minimal]
A curve $C/\mathbb{F}_q$ is {\em minimal} if 
\begin{equation*}
\#C(\mathbb{F}_q) = 1+q-2g\sqrt{q}.
\end{equation*}
\end{definition}
\begin{definition}[Maximal]
A curve $C/\mathbb{F}_q$ is {\em maximal} if 
\begin{equation*}
\#C(\mathbb{F}_q) = 1+q+2g\sqrt{q}.
\end{equation*}
\end{definition}

\begin{note}[{\cite[page 22]{MR0027151}}, {\cite[page 69]{MR0029522}}]\label{rmk:nwnlemma}
The curve C is maximal over $\mathbb{F}_q$ (resp.\ minimal over $\mathbb{F}_q$) if and only if all its normalized Weil numbers are -1 (resp.\ 1) over $\mathbb{F}_q$.
\end{note}

In the following remark, we use the notation that $\zeta_k$ is the primitive $k^{\text{th}}$ root of unity $e^{\frac{2\pi i}{k}}$. Notice that there is a power $s$ such that $\zeta^s_k = -1$ if and only if $k$ is even.
\begin{lemma}
Let $C$ be a supersingular curve over $\mathbb{F}_q$. Suppose the normalized Weil numbers of $C/\mathbb{F}_q$ are of the form $\zeta^{t_1}_{k_1}, \ldots, \zeta^{t_{2g}}_{k_{2g}}$. Assume $\gcd(k_i,t_i) = 1$. The curve $C$ is maximal over $\mathbb{F}_{q^{r}}$ if and only if
\begin{itemize}
\item there exists $s \geq 1$ and $b_i$ odd, such that $k_i = 2^s(b_i)$
\item and $r$ is an odd multiple of {}$2^{s-1}{\rm lcm}(b_1, \ldots, b_n)$.
\end{itemize}
\end{lemma}

\begin{proof} 
Assume $C$ is maximal over $\mathbb{F}_{q^r}$. By Remark \ref{rmk:nwnlemma}, the curve $C$ is maximal over $\mathbb{F}_{q^r}$ if and only if $\zeta^{rt_i}_{k_i} = -1$ for all $i$. Consequently, $k_i$ is even for all $i$. Thus $k_i = 2^{s_i} b_i$ for some positive integer $s_i$ and odd integer $b_i$. The condition $\zeta^{rt_i}_{k_i} = -1$ for all $i$ implies that there exists an $s$ such that $s=s_i$ for all $i$ and $r$ is an odd multiple of $2^{s-1}\text{lcm}(b_1, \ldots, b_n)$.

For the converse, the conditions imply that the normalized Weil numbers of $C$ over $\mathbb{F}_{q^r}$ are all $-1$.
\end{proof}

\section{Curve maps and covers}
\subsection{Aoki's Curve}
Let $\alpha = (a,b,c) \in \mathbb{N}^3$ with $a+b+c = m$. Note that $S_3$, the symmetric group on three letters, acts on $\alpha$ by permuting the coordinates. For $\sigma \in S_3$ we denote the action by $\alpha^\sigma$. We say two triples $\alpha = (a_1,a_2,a_3)$ and $\beta = (b_1,b_2,b_3)$ are equivalent, denoted $\alpha \approx \beta$, if there exist elements $t \in (\mathbb{Z}/m)^*$ and $\sigma \in S_3$ such that 
\begin{align*}
(a_1,a_2,a_3) \equiv (tb_{\sigma(1)},tb_{\sigma(2)},tb_{\sigma(3)}) \bmod m.
\end{align*}

In \cite{Aoki1} and \cite{Aoki2}, Aoki studies curves of the form
\begin{equation*}
D_\alpha: v^{m}=(-1)^{c}u^{a}(1-u)^{b}.
\end{equation*}
He provides the following conditions for when $D_\alpha$ is supersingular.
\begin{theorem}[{\cite[Theorem 1.1]{Aoki2}}] \label{Thm:Aoki1.1} 
The curve $D_\alpha$ is supersingular over $\mathbb{F}_{p^{r}}$ if and only if at least one of the following conditions holds:
\begin{itemize}
\item $p^{i} \equiv -1 \bmod m$ for some $i$.
\item $\alpha =(a,b,c) \approx (1,-p^{i},p^{i}-1)$ for some integer $i$ such that $d = \gcd(p^{i}-1,m) > 1$ and $p^{j} \equiv -1 \bmod \frac{m}{d}$ for some integer $j$. 
\end{itemize}\end{theorem}

\subsection{Covers of $H_{n,\ell}$ by $\mathcal{F}_m$}\label{sec:covers}
In Section \ref{sec:curves}, we noted that the Hurwitz curve $H_{n,\ell}$ is covered by the Fermat curve $\mathcal{F}_m$ where $m = n^2-n\ell+\ell^2$. On an affine patch the Fermat and Hurwitz curves are given by the following equations
\begin{align*}
\mathcal{F}_m: u^m + v^m + 1 &= 0 \\
H_{n,l}: x^n y^\ell + y^n + x^\ell &= 0. \\
\end{align*}
Then the following covering map is provided by \cite[Lemma 4.1]{AKT}
\begin{align*}
\phi: \mathcal{F}_m &\to H_{n,\ell} \\
(u, v)&\mapsto (u^{n}v^{-l}, u^{l}v^{n-l}). 
\end{align*}

Furthermore, it is known that $\mathcal{F}_m$ is supersingular over $\mathbb{F}_p$ if and only if $p^i \equiv -1 \bmod m$ for some integer $i$ \cite[Prop.\ 3.10]{shiodakatsura}. See also \cite[Theorem 3.5]{Yui}. In \cite[Theorem 5]{Tafazolian} it is shown that $\mathcal{F}_m$ is maximal over $\mathbb{F}_{p^{2i}}$ if and only if $p^{i}\equiv -1 \bmod m$.
\begin{note} \label{serre}
If $X \to Y$ is a covering of curves defined over $\mathbb{F}_{p^r}$, then the normalized Weil numbers of $Y/\mathbb{F}_{p^r}$ are a subset of the normalized Weil numbers of $X/\mathbb{F}_{p^r}$, see \cite{serre}.
\end{note}
Thus when a covering curve is supersingular (or maximal or minimal) the curve it covers is as well. 

\subsection{A Birational Transformations}
In \cite{Carbonne}, Bennama and Carbonne show that $H_{n,\ell}$ is isomorphic to a curve with affine equation
\begin{equation} \label{Eq:Carbonne Hurwitz}
y'^m=x'^{\lambda}(x'-1)
\end{equation}
via the following variable change. Suppose $1 \leq \ell < n$ and $\gcd(n,\ell)=1$. Then there exist integers $\theta$ and $\delta$ such that $1 \leq \theta \leq \ell$, $1 \leq \delta \leq n-1$, and $n \theta -\delta \ell =1$. Let $\lambda = \delta n - \theta (n-\ell)$ and $m=n^2-n\ell+\ell^2$. The birational transformation is as follows
\begin{equation*}
\begin{gathered}
\begin{cases}
	x=(-x')^{-\delta}((-1)^{\lambda}y')^n\\
    y=(-x')^{-\theta}((-1)^{\lambda}y')^\ell
\end{cases}
\end{gathered}
\end{equation*}
\centerline{\text{and}}
\begin{equation*}
\begin{gathered}
\begin{cases}
	x'=-x^\ell y^{-n}\\
    y'=(-1)^{\lambda}x^{\theta}y^{-\delta}.
\end{cases}
\end{gathered}
\end{equation*}
Equation \eqref{Eq:Carbonne Hurwitz} is very similar to the equation for $D_\alpha$ that Aoki studies but there are small differences. The following argument shows that these can be reconciled. Consequently, this variable change can be used to apply Aoki's results to Hurwitz curves.

Notice that equation \eqref{Eq:Carbonne Hurwitz} is divisible by $(x'-1)$ while Aoki studies curves whose equation contains a $(1-x')$ factor. Aoki requires that $a+b+c=m$ so the exponent on the negative sign is important. Inspecting equation \eqref{Eq:Carbonne Hurwitz} we see that $m$ will always be odd since $(n,\ell)=1$. Consequently, this negative sign is not an issue. Since $m$ is always odd we can replace $v$ with $-v$. This choice allows us to pick $c = m-a-b$. Then $b = 1$ and $a = \lambda$. 

\section{Supersingular Hurwitz Curves}
We arrive at explicit conditions on supersingularity for $H_{n,\ell}$ when $n$ and $\ell$ are relatively prime. We use results from \cite{Carbonne} and \cite{Aoki1} to accomplish this. We will be using affine equations for the Hurwitz curve in this section. 
\begin{lemma} \label{Lem:SS iff 2 conditions}
If $n$ and $\ell$ are relatively prime then $x^{n}y^{\ell}+y^{n}+x^{\ell} = 0$ is supersingular over $\mathbb{F}_{p}$ if and only if at least one of the following conditions holds.
\begin{enumerate}
\item  There exists $i\in\mathbb{Z}_{>0}$ such that $p^{i} \equiv -1 \bmod m$.

(In this case the Fermat curve covering the Hurwitz curve is maximal over $\mathbb{F}_{p^{2i}}$.)

\item There exists $i\in\mathbb{Z}_{>0}$ with $d = (p^{i}-1,m) > 1$ such that 
\begin{align*}
(\delta(n-\ell)+\ell\theta-1,1,-(\delta(n-\ell)+\ell\theta)) \approx (1,-p^{i},p^{i}-1)
\end{align*} 
and $p^{j} \equiv -1 \bmod(\frac{m}{d})$ for some integer $j$.
\end{enumerate}
\end{lemma}

\begin{proof}
We use the variable substitution from \cite{Carbonne} to apply Aoki's results to Hurwitz curves. We use the substitutions:
\begin{itemize}
\item $m=n^{2}-n\ell+\ell^{2}$,
\item $a=\lambda=\delta(n-\ell)+\ell\theta-1$,
\item $b=1$,
\item $c=m-(\delta(n-\ell)+\ell\theta)$.
\end{itemize}
Combining these with Aoki's results completes the proof.
\end{proof}
\begin{note} \label{Lem:(n,l,m)=1}
If $n$ and $\ell$ are relatively prime, then $n$ and $\ell$ are relatively prime to $n^{2}-n\ell+\ell^2$.
\end{note}
\begin{theorem} \label{Thm:ss/Fp}
Suppose $n$ and $\ell$ are relatively prime and $m = n^2 - n\ell + \ell^2$. Then $H_{n,\ell}$ is supersingular over $\mathbb{F}_p$ if and only if $p^{i} \equiv -1 \bmod m$ for some positive integer $i$.
\end{theorem}

\begin{proof}
If $p^i \equiv -1 \bmod m$ for some positive integer $i$, then $\mathcal{F}_m$ is supersingular over $\mathbb{F}_p$ by \cite[Prop.\ 3.10]{shiodakatsura}. Recall from section \ref{sec:covers} that $\mathcal{F}_m$ covers $H_{n,\ell}$. 
Thus $H_{n,\ell}$ is supersingular over $\mathbb{F}_p$ by Remark \ref{serre}.

Suppose $H_{n,\ell}$ is supersingular over $\mathbb{F}_p$. By Lemma \ref{Lem:SS iff 2 conditions} it is enough to show condition 2 in Lemma \ref{Lem:SS iff 2 conditions} can not happen. We begin by simplifying it using the substitution $\theta=\frac{1+\ell\delta}{n}$ and reducing modulo $m$ to show that condition 2 is equivalent to $(\frac{\ell}{n}-1,1,-\frac{\ell}{n})\approx(1,-p^{i},p^{i}-1)$ for some $i$ such that $d = (p^{i}-1,m) > 1$ and $p^{j} \equiv -1 \mod (\frac{m}{d})$ for some integer $j$. Recall that $\alpha \approx \alpha' \text{ if } \alpha=t\alpha'^\sigma$ for some $t \in (\mathbb{Z}/m)^{*}$ and $\sigma\in S_3$. We will show that $p^i-1$ and $m$ are relatively prime. We label the three coordinates of $(\frac{\ell}{n}-1,1,-\frac{\ell}{n})$ as $(a,b,c)$ and the three coordinates of $(1,-p^{i},p^{i}-1)$ as $(A,B,C)$. 

The proof will address six cases accounting for the orbit of $(A,B,C)$ under the action of $S_3$. In each case we will show that $\gcd(p^{i}-1,m)=1$. Specifically, we show $d=1$ by taking these congruences modulo $d$. By Remark \ref{Lem:(n,l,m)=1} we know that $n^{-1}$ exists modulo $m$ and modulo $d$. Finally, note that $\frac{\ell}{n}$ is relatively prime to $d$.

\begin{itemize}
\item $(a,b,c)\equiv t(A,B,C) \bmod m$: Comparing $c$ and $tC$ yields 
\begin{align*}
-\frac{\ell}{n}\equiv t(p^{i}-1) \bmod m.
\end{align*}
Consequently, $\frac{\ell}{n}\equiv 0 \bmod d$. Therefore, $d = 1$. 

\item {$(a,b,c)\equiv t(B,A,C) \bmod m$}: Comparing $a$ with $tB$ and $b$ with $tA$ yields 
\begin{align*}
\frac{\ell}{n}-1	&\equiv -tp^i \bmod m \\
1					&\equiv t \bmod m.
\end{align*}
Substituting we have $\frac{\ell}{n}\equiv p^{i}-1 \bmod m$.  Reducing modulo $d$ produces $\frac{\ell}{n}\equiv 0 \bmod d$, thus $d = 1$.
\item {$(a,b,c)\equiv t(A,C,B) \bmod m$}: Comparing $b$ and $tC$ yields 
\begin{align*}
-\frac{\ell}{n}\equiv t(p^{i}-1) \bmod m.
\end{align*}
This is identical to the first case. 
\item {$(a,b,c)\equiv t(C,B,A) \bmod m$}: Comparing $a$ and $tC$ yields 
\begin{align*}
\frac{\ell}{n} - 1 \equiv t(p^i -1) \bmod m.
\end{align*}
Thus $\frac{\ell}{n} -1 \equiv 0 \bmod d$. Recall by the definition of $m$ and selection of $d$, we have $d \mid n^2 - n\ell + \ell^2$. Hence, $d$ divides $1 - \frac{\ell}{n} + (\frac{\ell}{n})^2$. We conclude $d|(\frac{\ell}{n})$, thus $d = 1$.
\item {$(a,b,c)\equiv t(C,A,B) \bmod m$}: Comparing $b$ with $tA$ and $c$ with $tB$ yields 
\begin{align*}
1 				&\equiv t \bmod m \\
\frac{\ell}{n}	&\equiv tp^i \bmod m.
\end{align*}
This case is completed as in the previous case. 
\item {$(a,b,c)\equiv t(B,C,A) \bmod m$}: Comparing $b$ with $tC$ yields 
\begin{align*}
1 \equiv t(p^{i}-1) \bmod m.
\end{align*}
Modulo $d$ this reduces to $1\equiv 0 \bmod d$. Therefore, $d = 1$.
\end{itemize}
\end{proof}

\begin{note}\label{rmk:TT}
There is a family of Hurwitz type curves with affine equations $\mathcal{C}_{a_1,a_2,n_1,n_2}: x^{n_1} y^{a_1} + y^{n_2} + x^{a_2} = 0$. Set $\delta = a_1 a_2 - a_2 n_2 + n_1 n_2$. When $q = p^r$ is coprime to $\delta$ then the curve $C_{a_1,a_2,n_1,n_2}$ is $\mathbb{F}_q$-covered by the Fermat curve $\mathcal{F}_\delta$ of degree $\delta$. In \cite[Theorem 2.9]{MR3562537} Tafazolian and Torres show that under certain numerical conditions the statements
\begin{itemize}
	\item the Fermat curve $\mathcal{F}_\delta$ is maximal over $\mathbb{F}_{q^2}$;
	\item the Hurwitz type curve $\mathcal{C}_{1,a_2,n_1,n_2}$ is maximal over $\mathbb{F}_{q^2}$;
	\item and $q+1 \equiv 0 \mod \delta$
\end{itemize}
are all equivalent.

The Hurwitz type curve $\mathcal{C}_{\ell,\ell,n,n}$ is the Hurwitz curve $H_{n,\ell}$. Thus in the case that $\ell = a_1=a_2$ and $n = n_1 = n_2$, Theorem \ref{Thm:ss/Fp} generalizes \cite[Theorem~2.9]{MR3562537}.
\end{note}

\begin{note}
Consider the family of curves with affine equations
\begin{align*}
N_{a_1,a_2,n_1,n_2} : x^{n_1} y^{a_1} + k_1 y^{n_2} + k_2 x^{a_2} = 0
\end{align*}
over $\mathbb{F}_{p^r}$ with $k_1,k_2 \in (\mathbb{F}_q)^{*}$, $n_1 \geq a_1$, $n_1 + a_1 > a_2$, $n_1 + a_1 > n_2$, if $n_1 = a_1$ then $n_2 \geq a_2$, and $p \nmid \gcd(a_1, a_2, n_1, n_2)$. Set $d = \gcd(a_1,a_2,n_1,n_2)$ and $\delta$ as in Remark \ref{rmk:TT}. Recall the definition of $n_p$ in Definition \ref{defn:val}. With these assumptions \cite[Theorem 4.12]{MR3475542} shows that if $(\delta/d)_p$ divides $q + 1$ then $N_{a_1,a_2,n_1,n_2}$ is maximal over $\mathbb{F}_q$ and if $N_{a_1,a_2,n_1,n_2}$ is maximal over $\mathbb{F}_{q^2}$ then $(\delta/d)_p$ divides $q^2 + 1$.

Note $N_{\ell,\ell,n,n} = H_{n,\ell}$. Thus Theorem \ref{Thm:ss/Fp} generalizes \cite[Theorem 4.12]{MR3475542} when $a_1 = a_2 = \ell$ and $n_1 = n_2 = n$.
\end{note}

\begin{corollary} \label{Cor:ss/Fp-max}
If $n$ and $\ell$ are relatively prime and $H_{n,\ell}$ is supersingular over $\mathbb{F}_p$, then it will be maximal over $\mathbb{F}_{p^{2i}}$ where $i$ is the same as in Theorem \ref{Thm:ss/Fp}.
\end{corollary}
\begin{proof}
By Theorem \ref{Thm:ss/Fp}, if $H_{n,\ell}$ is supersingular over $\mathbb{F}_p$, then $p^{i} \equiv -1 \bmod m$ for some $i$. By the results of \cite{Tafazolian}, this implies $\mathcal{F}_m$ will be maximal over $\mathbb{F}_{p^{2i}}$. Since  $\mathcal{F}_m$ covers $H_{n,\ell}$, this implies $H_{n,\ell}$ will also be maximal over $\mathbb{F}_{p^{2i}}$.  
\end{proof}

A priori, if $H_{n,\ell}$ is supersingular (or maximal or minimal) over $\mathbb{F}_p$ then $\mathcal{F}_m$ may not be because it has more normalized Weil numbers.
\begin{corollary} \label{cor:SSHurwitz Covered by SSFermat}
If $n$ and $\ell$ are relatively prime and $H_{n,\ell}$ is supersingular over $\mathbb{F}_p$, then $\mathcal{F}_m$ is supersingular over $\mathbb{F}_p$.
\end{corollary}
\begin{proof}
If $H_{n,\ell}$ supersingular over $\mathbb{F}_p$ and $\gcd(n,\ell) = 1$, Theorem \ref{Thm:ss/Fp} shows the existence of positive integer $i$ such that $p^i \equiv -1 \bmod m$. Then by \cite[Proposition~3.10]{shiodakatsura}, $\mathcal{F}_m$ is supersingular over $\mathbb{F}_p$.
\end{proof}

Partial results are known for when a Hurwitz curve is maximal. 
\begin{theorem}[{\cite[Theorem 3.1]{AKT}}]
Let $\ell = 1$. The curve $H_{n,1}$ is maximal over $\mathbb{F}_{q^{2j}}$ if and only if $p^j \equiv -1 \bmod m$ for some positive integer $j$.
\end{theorem}
\begin{theorem}[{\cite[Theorem 4.5]{AKT}}]
Assume that $\gcd(n,\ell) = 1$ and $m$ is prime. Then $H_{n,\ell}$ is maximal over $\mathbb{F}_{p^{2j}}$ if and only if $p^j \equiv -1 \bmod m$ for some positive integer $j$.
\end{theorem}
Note that the key property used in \cite{AKT} is the existence of some positive integer $j$ such that
\begin{equation}\label{keyprop}
p^j \equiv -1 \bmod m. 
\end{equation}
\begin{note}
Under the requirements $\ell = 1$, or $\gcd(n,\ell) = 1$ and $m$ prime, the results in \cite{AKT} and \cite[Theorem 5]{Tafazolian} show that $\mathcal{F}_m$ is maximal over $\mathbb{F}_{q^2}$ if and only if $H_{n,\ell}$ is maximal over $\mathbb{F}_{q^2}$.
\end{note}
We consider the case when $H_{n,\ell}$ and $\mathcal{F}_m$ are minimal.
\begin{corollary} \label{Tmin}
If $\ell=1$, or $n$ and $\ell$ are relatively prime and $m$ is prime, $H_{n,\ell}$ is minimal over $\mathbb{F}_{p^{4i}}$ if and only if $\mathcal{F}_m$ is minimal over $\mathbb{F}_{p^{4i}}$.
\end{corollary}
\begin{proof}
First suppose $\mathcal{F}_m$ is minimal over $\mathbb{F}_{p^{4i}}$ with set $N$ of normalized Weil numbers. Then the normalized Weil numbers of $H_{n,\ell}$ are a subset of $N$. Thus $H_{n,\ell}$ will also be minimal over $\mathbb{F}_{p^{4i}}$.

Now assume $H_{n,\ell}$ is minimal over $\mathbb{F}_{p^{4i}}$. Minimality implies supersingularity, thus $H_{n,\ell}$ must also be supersingular. By Theorem \ref{Thm:ss/Fp} supersingularity of $H_{n,\ell}$ over $\mathbb{F}_p$ implies $p^j \equiv -1 \bmod m$ for some positive integer $j$. Choose a minimal such $j$. Then Corollary \ref{Cor:ss/Fp-max} shows $H_{n,\ell}$ is maximal over $\mathbb{F}_{p^{2j}}$ thus minimal over $\mathbb{F}_{p^{4j}}$. Minimality of $j$ implies that $\mathbb{F}_{p^{4j}}$ is a subfield of $\mathbb{F}_{p^{4i}}$. Consequently, $j\mid i$.

Now, by \cite{AKT} $p^j \equiv -1 \bmod m$ implies that $\mathcal{F}_m$ is maximal over $\mathbb{F}_{p^{2j}}$. Hence, $\mathcal{F}_m$ is minimal over $\mathbb{F}_{p^{4j}}$. Because $j \mid i$, $\mathcal{F}_m$ is minimal over $\mathcal{F}_{p^{4i}}$.
\end{proof}

\begin{note} \label{Rmk:HurwitzMax/FermatNot}
The curve $H_{3,3}$ is maximal over $\mathbb{F}_{5^{2}}$ but $\mathcal{F}_{9}$ is not. The above theorems show a supersingular Hurwitz curve and its covering Fermat curve will both be maximal over $\mathbb{F}_{p^{2i}}$. This does not imply that the Fermat curve will always be maximal over the same field extension that the Hurwitz curve is. The Hurwitz curve could also be maximal over  $\mathbb{F}_{p^{2j}}$ where $j\mid i$ with $i/j$ odd. In this case the Fermat curve may not be maximal over this field because it has a higher genus. Unfortunately our example of this does not have $n$ and $\ell$ being relatively prime. It is difficult to find an example with $n$ and $\ell$ relatively prime, as the genera of Hurwitz curves grow quickly causing the point counts to become computationally expensive.
\end{note}

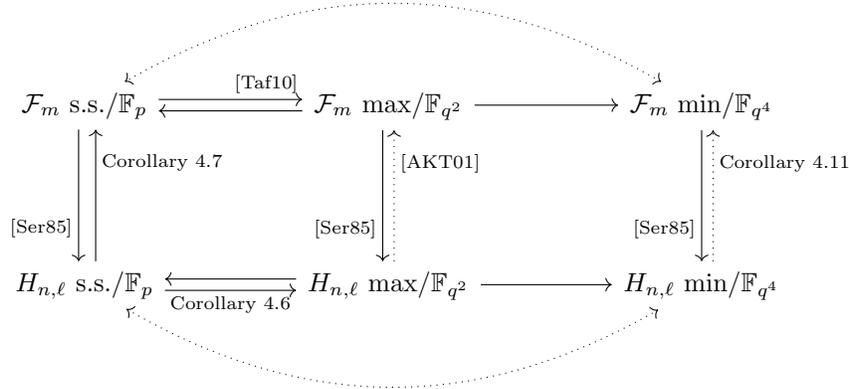
\begin{figure}[!ht]
\centering
\usetikzlibrary{decorations.pathmorphing}
\begin{tikzcd}[column sep=5em,row sep=5em]
\text{$\mathcal{F}_m$ s.s.$/\mathbb{F}_p$}  \arrow[r, "\text{\cite{Tafazolian}}" near end, shift left=.5ex] \arrow[d, shift right=.5ex, "\text{\cite{serre}}" swap, near end] 
	& \text{$\mathcal{F}_m$ max$/\mathbb{F}_{q^2}$} \arrow[r] \arrow[l, shift left=.5ex] \arrow[d, "\text{\cite{serre}}" near end, swap, shift right=.5ex]
	& \text{$\mathcal{F}_m$ min$/\mathbb{F}_{q^4}$} \arrow[d, "\text{\cite{serre}}" near end, swap] \arrow[ulld, leftrightarrow, dotted, bend right] \\
	\text{$H_{n,\ell}$ s.s.$/\mathbb{F}_p$} \arrow[r, "\text{Corollary \ref{Cor:ss/Fp-max}}" swap, shift right=.5ex] \arrow[u, "\text{Corollary \ref{cor:SSHurwitz Covered by SSFermat}}" swap, near end, shift right=1ex] \arrow[drru, leftrightarrow, dotted, bend right]
	& \text{$H_{n,\ell}$ max$/\mathbb{F}_{q^2}$} \arrow[r] \arrow[l, shift right=.5ex] \arrow[u, dotted, "\text{\cite{AKT}}" near end, swap, shift right=.5ex]
    & \text{$H_{n,\ell}$ min$/\mathbb{F}_{q^4}$} \arrow[u, dotted, "\text{Corollary \ref{Tmin}}" swap, near end, shift right=1ex]
\end{tikzcd}
\caption{Current results regarding supersingularity, minimality, and maximality of Hurwitz and Fermat curves.}
\label{fig:theory}
\end{figure}
Figure \ref{fig:theory} illustrates how the current theory fits together. The straight, dotted arrows are under the conditions $\ell = 1$, or $\gcd(n,\ell) = 1$ and $m$ prime. The notation max/$\mathbb{F}_{q^2}$ means, for some power $q$ of $p$, the curve is maximal over $\mathbb{F}_{q^2}$. If a curve is maximal over $\mathbb{F}_{q^2}$ then it is minimal over $\mathbb{F}_{q^4}$. The curved arrows show that under appropriate conditions a Hurwitz or Fermat curve is supersingular if and only if it is minimal over some field extension. Corollaries~\ref{Cor:ss/Fp-max} and \ref{cor:SSHurwitz Covered by SSFermat} are under the condition that $\gcd(n,\ell) = 1$, while \cite{AKT} and Corollary \ref{Tmin} are under the condition that $\ell = 1$, or $\gcd(n,\ell) = 1$ and $m$ is prime.

\section{Which Genera Occur and Data}
Here we provide information about which genera occur for Hurwitz curves and provide a classification of supersingular Hurwitz curves having genus less than $5$ defined over ${\mathbb F}_p$ when $p < 37$. 

Recall that the genus of the Hurwitz curve $H_{n,\ell}$ has the following equation
\begin{equation*}
g=\dfrac{n^2-n\ell+\ell^2 - 3\gcd(n,\ell) + 2}{2}.
\end{equation*}

From this, it can be seen that the genus is determined by the quadratic form $q(x,y)=x^2-xy+y^2$ and $\gcd(x,y)$. 
In this section, we provide information about which genera can appear as a result of these equations.

%
%

\begin{theorem}[{\cite[Vol.\ II, pages 310-314]{fermat}}]\label{Thm:PrimeDecomp}
The equation $m=x^2-xy+y^2$ has solutions $x,y \in \mathbb{Z}$ if and only if for every prime $p$ in  the prime decomposition of $m$, either $p \equiv 0,1 \bmod 3$ or $p$ is raised to an even power.
\end{theorem}

There is no restriction in Theorem \ref{Thm:PrimeDecomp} on what the values $x$ and $y$ are. However, for Hurwitz curves we require $n$ and $\ell$ to be positive. The question remains as to when the equation $m=q(x,y)$ has solutions in the positive integers. To solve this we study the following automorphisms of $q(x,y)=m.$
\begin{gather*}
\begin{cases}
f: \ \mathbb{Z}^2 \rightarrow \mathbb{Z}^2 \mid f(x,y) \mapsto (y,x)\\
g: \ \mathbb{Z}^2 \rightarrow \mathbb{Z}^2 \mid g(x,y) \mapsto (-x,-y)\\
\varphi: \ \mathbb{Z}^2 \rightarrow \mathbb{Z}^2 \mid \varphi(x,y) \mapsto (x,x-y)\\
I: \ \mathbb{Z}^2 \rightarrow \mathbb{Z}^2 \mid I(x,y) \mapsto (x,y)
\end{cases}
\end{gather*}
To see that $\varphi(x,y)$ is an automorphism, compute the following
\begin{align*} q\circ\varphi(x,y) &= x^2-x(x-y)+(x-y)^2 \\
&= x^2 -x^2+xy+x^2-2xy+y^2 \\
&= x^2-xy+y^2 \\
&= q(x,y).
\end{align*}
\begin{corollary} \label{comorollary}
If the equation $m=q(x,y)$ has a solution $(x,y)\in \mathbb{Z}^2$ then there is a solution with $(x',y')\in \mathbb{N}^2$.
\end{corollary}
\begin{proof}
We separate into cases, depending on the values of $x$ and $y$.
\begin{enumerate}
\item If both $x$ and $y$ are negative, then $g(x,y) = (-x,-y) \in \mathbb{N}^2$.
\item If $y$ negative and $x$ positive, then $\varphi(x,y) = (x,x-y) \in \mathbb{N}^2$.
\item If $x$ negative and $y$ positive, then $\varphi(f(x,y)) = (y,y-x) \in \mathbb{N}^2$.
\item If $x$ is 0, then $\varphi\circ f (0,y) = (y,y)$ and if $y$ is 0, then $\varphi(y,0) = (y,y)$.
\end{enumerate}
\end{proof}

By counting points and using Lemma \ref{Lem:ZetaCoeffs} we computed, using \cite{sage}, the $L$-polynomials and normalized Weil numbers of many supersingular Hurwitz curves over $\mathbb{F}_p$. When $n$ and $\ell$ are not relatively prime, it is possible that certain points of the equation for $H_{n,\ell}$ are singular. Resolving these singularities requires taking a field extension of $\mathbb{F}_p$. To adjust for this we check if $q\equiv 1\bmod\gcd(n,\ell)$ and count the multiplicities of singular points. This gives the correct point counts to compute the $L$-polynomial of the normalization of the equation. The table has all supersingular Hurwitz curves $H_{n,\ell}$ of genus less than 5 for primes less than 37. The table also includes some curves of genus 6.

\begin{table}[!ht]
\small
\begin{tabular}{|l|l|l|l|l|l|}
    \hline
      \textbf{n} & \textbf{l} & \textbf{p} & \textbf{g} & \textbf{L-Polynomial} & \textbf{NWNs (multiplicity)}\\ 
   		\hline
        2 & 1 & 5 & 1 & $5T^2 + 1$ & i, -i\\
        2 & 1 & 11 & 1 & $11T^2 + 1$ & i, -i\\
        2 & 1 & 17 & 1 & $17T^2 + 1$ & i, -i\\
        2 & 1 & 23 & 1 & $23T^2 + 1$ & i, -i\\
        2 & 1 & 29 & 1 & $29T^2 + 1$ & i, -i \\
        \hline
        3 & 3 & 5 & 1 & $5T^2 + 1$ & i, -i\\
        3 & 3 & 11 & 1 & $11T^2 + 1$ & i, -i\\
        3 & 3 & 17 & 1 & $17T^2 + 1$ & i, -i\\
        3 & 3 & 23 & 1 & $23T^2 + 1$ & i, -i\\
        3 & 3 & 29 & 1 & $29T^2 + 1$ & i, -i\\
        \hline
        3 & 1 & 3 & 3 & $27T^6 + 1$ & i,-i, $\zeta_{12}, \zeta_{12}^5, \zeta_{12}^7, \zeta_{12}^{11}$\\
        3 & 1 & 5 & 3 & $125T^6 + 1$ & i,-i, $\zeta_{12}, \zeta_{12}^5, \zeta_{12}^7, \zeta_{12}^{11}$\\
        3 & 1 & 13 & 3 & $2197T^6 + 507T^4 + 39T^2 + 1$ & i(3), -i(3)\\
        3 & 1 & 17 & 3 & $4913T^6 + 1$ & i, -i, $\zeta_{12}, \zeta_{12}^5, \zeta_{12}^7, \zeta_{12}^{11}$\\
        3 & 1 & 19 & 3 & $6859T^6 + 1$ & i, -i, $\zeta_{12}, \zeta_{12}^5, \zeta_{12}^7, \zeta_{12}^{11}$\\
        3 & 1 & 31 & 3 & $29791T^6 + 1$ & i, -i, $\zeta_{12}, \zeta_{12}^5, \zeta_{12}^7, \zeta_{12}^{11}$\\
        \hline
        3 & 2 & 3 & 3 & $27T^6 + 1$ & i,-i, $\zeta_{12}, \zeta_{12}^5, \zeta_{12}^7, \zeta_{12}^{11}$\\
        3 & 2 & 5 & 3 & $125T^6 + 1$ & i,-i, $\zeta_{12}, \zeta_{12}^5, \zeta_{12}^7, \zeta_{12}^{11}$\\
        3 & 2 & 13 & 3 & $2197T^6 + 507T^4 + 39T^2 + 1$ & i(3), -i(3)\\
        3 & 2 & 17 & 3 & $4913T^6 + 1$ & i, -i, $\zeta_{12}, \zeta_{12}^5, \zeta_{12}^7, \zeta_{12}^{11}$\\
        3 & 2 & 19 & 3 & $6859T^6 + 1$ & i, -i, $\zeta_{12}, \zeta_{12}^5, \zeta_{12}^7, \zeta_{12}^{11}$\\
        3 & 2 & 31 & 3 & $29791T^6 + 1$ & i, -i, $\zeta_{12}, \zeta_{12}^5, \zeta_{12}^7, \zeta_{12}^{11}$ \\
        \hline
        4 & 2 & 5 & 4 & $625T^8 + 500T^6 + 150T^4 + 20T^2 + 1$ &i(4), -i(4)\\
        4 & 2 & 17 & 4 & $83521T^8 + 19652T^6 + 1734T^4 + 68T^2 + 1$ &i(4), -i(4)\\
        4 & 2 & 29 & 4 & $707281T^8 + 97556T^6 + 5046T^4 + 116T^2 + 1$ &i(4), -i(4)\\
        \hline
        4 & 1 & 5 & 6 & $15625T^{12} + 1875T^8 + 75T^4 + 1$ & $\zeta_{8}(3), \zeta_{8}^3(3), \zeta_{8}^5(3), \zeta_{8}^7(3)$\\
        \hline
        4 & 3 & 5 & 6 & $15625T^{12} + 1875T^8 + 75T^4 + 1$ & $\zeta_{8}(3), \zeta_{8}^3(3), \zeta_{8}^5(3), \zeta_{8}^7(3)$\\ 
        \hline
        5 & 5 & 3 & 6 & $729T^{12} + 243T^8 + 27T^4 + 1$ &$\zeta_{8}(3), \zeta_{8}^3(3), \zeta_{8}^5(3), \zeta_{8}^7(3)$\\
        5 & 5 & 7 & 6 & $117649T^{12} + 7203T^8 + 147T^4 + 1$ & $\zeta_{8}(3), \zeta_{8}^3(3), \zeta_{8}^5(3), \zeta_{8}^7(3)$\\
        5 & 5 & 13 & 6 & $4826809T^{12} + 85683T^8 + 507T^4 + 1$ & $\zeta_{8}(3), \zeta_{8}^3(3), \zeta_{8}^5(3), \zeta_{8}^7(3)$\\
        \hline
\end{tabular}
\caption{Supersingular Hurwitz curves in characteristic $p< 37$ with genus $< 5$.}
\end{table}

\bibliography{bibliography}

\end{document}